\documentclass[draft]{amsart}

\usepackage{amsmath}
\usepackage{amssymb}
\usepackage{amsopn}
\usepackage{amsthm}
\usepackage[latin1]{inputenc}
\usepackage[british]{babel}

\usepackage{tikz}
\usetikzlibrary{through}
\usepackage{pgffor}

\newtheorem{Thm}{Theorem}
\newtheorem{Cor}{Corollary}
\newtheorem{Lem}{Lemma}

\theoremstyle{definition}
\newtheorem{Def}{Definition}
\newtheorem{Ex}{Example}

\theoremstyle{remark}

\newcommand{\nat}{\mathbb{N}}

\DeclareMathOperator{\homology}{\mathrm H}         

\newcommand{\im}{\mathop{\rm Im}\nolimits}         
\newcommand{\kernel}{\mathop{\rm Ker}\nolimits}    
\newcommand{\cpx}[1]{#1_{\bullet}}                 

\newcommand{\digraph}[1]{\Gamma_{#1}}
\newcommand{\morsegraph}[2]{\Gamma_{#1}^{#2}}

\newcommand{\lcm}{\mathop{\rm lcm}\nolimits}       
\newcommand{\redpaths}[2]{[#1 \rightsquigarrow #2]}

\newcommand{\sgn}[2]{\mathop{\varepsilon(#1;#2)}\nolimits}

\newcommand{\initial}{\mathop{\rm in}\nolimits}

\newcommand{\component}[2]{#1_{\langle#2\rangle}}

\newcommand{\nf}{\mathop{\rm nf}\nolimits}

\newcommand{\crit}{\mathop{\rm crit}\nolimits}
\newcommand{\ncrit}{\mathop{\rm ncrit}\nolimits}
\newcommand{\ccrit}{\mathop{c{\rm \mbox{-}crit}}\nolimits}
\newcommand{\supp}{\mathop{\rm supp}\nolimits}
\renewcommand{\epsilon}{\varepsilon}
\renewcommand{\phi}{\varphi}
\newcommand{\canonicalsurj}{\eta}

\DeclareMathOperator{\alt}{\mathrm Alt}   

\newcommand{\card}[1]{|#1|}

\newcommand{\revlexlesseq}{\leq_{\mathrm{revlex}}}
\newcommand{\lexless}{<_{\mathrm{lex}}}

\newcommand{\resbasis}[2]{e_{#1} g_{#2}}           
\newcommand{\bigbasis}[3]{e_{#1} x^{#2} g_{#3}}    
\newcommand{\altbasis}[2]{e_{#1} #2}               
\newcommand{\altbigbasis}[3]{e_{#1} #2 #3}         
\newcommand{\modbasis}[1]{g_{#1}}                  

\title{Resolutions of modules with initially linear syzygies}

\author{Emil Sk\"oldberg}

\address{School of Mathematics, Statistics and Applied Mathematics\\
         National University of Ireland, Galway\\
         Ireland}

\email{emil.skoldberg@nuigalway.ie}
\urladdr{http://www.maths.nuigalway.ie/~emil/}
\date{\today}

\begin{document}

\maketitle

\begin{abstract}
  We introduce the class of modules with initially linear syzygies,
  a class including ideals with linear quotients, and study their minimal
  resolutions. Using a contracting homotopy for the resolutions, we
  see that the minimal resolution of a matroidal monomial ideal admits
  a DGA structure.
\end{abstract}

\section{Introduction}

Let $k$ be a field and $S = k[x_1,\dots,x_n]$, the polynomial ring over
$k$. In this paper we introduce the notion of modules with initially
linear syzygies, and study the structure of their minimal resolutions.
We show how to construct the minimal resolution of a module with
initially linear syzygies using discrete Morse theory,
and we then show that such modules are componentwise linear;
finally, at the end of the paper we investigate
multiplicative properties, and
we show that the minimal free resolution of
$S/I$ admits a differential graded algebra structure,
where $I$ is either a stable monomial ideal or a
squarefree matroidal ideal. The result
for stable ideals is not new, it has been shown by
Peeva~\cite{peeva:borelfixed}, but we obtain a simpler proof.
The result for squarefree matroidal ideals is new, however.
We finish the paper
by calculating the product on a part of the minimal resolution of the
ideal coming from the Fano matroid as an illustration of the results.

\section{Modules with initially linear syzygies}
\begin{Def}
  A presentation of a finitely generated graded $S$-module $M$,
  \begin{equation} \label{eq:in_lin_syz}
    0 \longrightarrow
    \kernel \canonicalsurj \longrightarrow
    \bigoplus_i S \cdot g_i  \stackrel{\canonicalsurj}{\longrightarrow}
    M \longrightarrow 0,
  \end{equation}
  is said to have \emph{initially linear syzygies} with respect to a
  term order $\prec$ on the free module  $\bigoplus_i S \cdot
  g_i$, if
  $\kernel \canonicalsurj \subseteq \bigoplus_i \mathfrak{m} \cdot g_i$,
  and if the initial module $\initial_{\prec}(\kernel \canonicalsurj)$ is
  generated by terms of the form $x_j g_i$.
\end{Def}
We will say that $M$ has initially linear syzygies if it
has such a presentation for some choice of generating set
$\{g_1, \ldots, g_n\}$ and term order $\prec$.

Ideals with initially linear syzygies generalise ideals
with linear quotients; let us recall that an ideal $I$ has linear
quotients if there are elements $f_1, f_2, \ldots, f_n$ such
that $I = (f_1, \ldots, f_n)$, and for each $1 < i \leq n$, the colon
ideal $(f_1,\dots,f_{i-1}):f_i$ is generated by linear forms.
It is not hard to see that a homogeneous ideal has linear quotients if
and only if it has initially linear syzygies with respect to a
position over term order, which is a term order on a free module
$\bigoplus_i S \cdot g_i$ such that $g_i < g_j$ implies that
$x^{\alpha}g_{i} < x^{\beta}g_{j}$ for all $\alpha$ and $\beta$.
Ideals with linear quotients in turn generalise shellable monomial
ideals, a concept introduced by Batzies and
Welker~\cite{batzies_welker:morse_cellular}.
A monomial ideal $I$ with minimal
monomial generators $m_1,\dots,m_t$ is \emph{shellable} if there is
a total order
``$\sqsubset$'' on $m_1,\dots,m_t$ such that for $m_j, m_i$ with
$m_j \sqsubset m_i$ there is an $m_k$ such that $m_k \sqsubset m_i$ and
$x_{g(m_k,m_i)} m_i = \lcm(m_k,m_i)$ divides $\lcm(m_j,m_i)$ for some
index $g(m_k,m_i)$. Thus a shellable monomial ideal has linear
quotients and then also initially linear syzygies.

The families of monomial ideals we are particularly interested in are
the stable ideals, and the squarefree matroidal ideals, which are
ideals where the supports of the minimal generators form the bases of a
matroid.


Since we will extensively use algebraic Morse theory, we will below
briefly review our terminology. For details on algebraic
Morse theory, see \cite{joellenbeck_welker:morse},
\cite{jonsson:graphs}, \cite{kozlov:morse} and \cite{skoldberg:morse}.
By a \emph{based complex} of $R$-modules we mean a chain complex
\( \cpx{K} \) of  \( R \)-modules together with a direct sum
decomposition  \( K_n = \bigoplus_{\alpha \in I_n} K_{\alpha} \) where
 \( \{I_n\} \) is a family of  mutually disjoint index sets.
For $f: \bigoplus_n K_n \rightarrow \bigoplus_n K_n$ a graded map, we write
$f_{\beta,\alpha}$ for the component of $f$ going from
$K_{\alpha}$ to $K_{\beta}$, and given a based complex \(\cpx{K}\) we
construct  a digraph \( \digraph{\cpx{K}} \) with vertex set \( V =
\bigcup_n I_n \) and with a directed  edge
\( \alpha \rightarrow \beta \) whenever the component
$d_{\beta,\alpha}$ is non-zero.

A \emph{partial matching  on a digraph $D=(V,E)$} is a subset
$A$ of the edges $E$ such that no vertex is incident  to
 more than one edge in $A$. In this situation we define
the new digraph $D^A=(V,E^A)$ to be the digraph obtained from $D$ by
reversing the direction of each arrow in $A$. Given the matching
$A$, we define the sets $A^{+}$, $A^{-}$ and $A^{0}$ by letting
$A^{+}$ be the set of vertices that are targets of a reversed
arrow from $A$; $A^{-}$ be the set of vertices that are sources
of a reversed arrow from $A$; and $A^{0}$ to be the vertices that
are not incident to an arrow from $A$.
We call a partial matching \(A\) supported on the digraph
\(\digraph{\cpx{K}} \) a \emph{Morse matching} if, for  each edge
\(\alpha \rightarrow \beta \) in \(A\), the corresponding component
\(d_{\beta,\alpha} \) is an isomorphism, and furthermore there is a well
founded partial order $\prec$ on each $I_n$ such that  $\gamma \prec
\alpha$ whenever there is a path $\alpha^{(n)} \rightarrow \beta
\rightarrow \gamma^{(n)}$ in $\morsegraph{\cpx{K}}{A}$.

\section{The minimal resolution}

In this section we start by observing that a finitely generated
$S$-module $M$ has a free resolution given by a two-sided Koszul
complex $\cpx{G}$. The modules in this resolution are not even finitely
generated, so it is far from being minimal. In this resolution we
see that we can find a matching that gives us a projection that allows
us to find the minimal resolution $\cpx{F}$ of $M$ as a direct summand
of $\cpx{G}$ in the case when $M$ has initially linear syzygies.
We then give a description of the differential in terms
of reductions following J\"ollenbeck and
Welker~\cite{joellenbeck_welker:morse}, and then
show that in some cases the differential is of Eliahou--Kervaire
type. We round off by showing that modules with initially linear
syzygies are componentwise linear.

Let $M$ be a finitely generated $S$-module,
let $V$ be the $k$-vector space with basis $e_1,\dots, e_n$ and
let $\cpx{F}$ be the chain complex with modules
$F_p = S \otimes \alt^p V \otimes M$.
For an element $e_{i_1} \wedge \cdots \wedge e_{i_d}$ of $\alt^d V $
with $I = \{i_1, \dots ,i_d\} \subseteq [n]$ and
$i_1 < \dots < i_d$, we will write
$e_{I}$, and we will also write $e_Im$ for the element
$1 \otimes e_I \otimes m$; as $m$ ranges over a $k$-basis of $M$,
 these elements obviously form an $S$-basis for $F_n$.
The differential
$d_n :   S \otimes \alt^n V \otimes M \longrightarrow
 S \otimes \alt^{n-1} V \otimes M$ is defined on the basis elements by
\begin{displaymath}
d(e_I m) =
\sum_{i \in I}  \sgn{i}{I}
(x_{i} \, e_{I \smallsetminus i} m -
  e_{I\smallsetminus i }  x_{i}m),
\end{displaymath}
where the sign $\sgn{i}{I}$ is defined by
\[
\sgn{i}{I} = (-1)^{\card{\{j \mid j \in I, j < i\}}}.
\]
\begin{Lem} \label{lemma:gen_koszul}
The complex $\cpx{F}$ is a free resolution of $M$.
\end{Lem}

\begin{proof}
It is obvious that $\homology_{0} (\cpx{F})  \simeq M$, so we just have to
prove that $\homology_{i} (\cpx{F}) = 0$ for $i \geq 1$.

For a $k$-vector space basis $B$ of $M$ we consider $\cpx{F}$ as a
based complex of $k$-vector spaces via the natural decomposition
\[
S \otimes \alt^{\bullet} V \otimes M \simeq
\bigoplus_{\alpha \in \nat^n, I \subseteq [n], m \in B}
k \cdot x^{\alpha} \, e_I m.
\]
For each $i$ we now define the following subset of the vertices in the
digraph $\digraph{\cpx{F}}$:
\[
V_i = \{x^{\alpha} \, e_I m \mid \deg x^{\alpha} + \card I = i\}.
\]
Now, construct a partial matching $E_{i}$ on the subgraph
$\digraph{\cpx{F}}|_{V_i}$ consisting of the edges
\[
E_i =
\{
x^{\alpha} \, e_I  m
\rightarrow
x^{\alpha}x_{i} \, e_{I\smallsetminus i} m
\mid
i = \min(\supp \alpha \cup I)
\}.
\]
It should be clear that if $\alpha \rightarrow \beta \in E_i$ for
$\alpha, \beta \in V_i$ then all $\gamma \in V_i$,
$\gamma \neq\beta$
with an edge
$\alpha \rightarrow \gamma$ are unmatched, so $E_i$ is a Morse
matching on $\digraph{\cpx{F}}|_{V_i}$, and since for all edges
$\alpha \rightarrow \beta$ with $\alpha \in V_i$ and $\beta \in V_j$ we
have $j \leq i$, we get from  \cite[Lemma 7]{skoldberg:morse} that
$\bigcup_{i} E_i$ is a Morse matching on $\digraph{\cpx{F}}$.
The claim of the lemma now follows from
\cite[Theorem 1]{skoldberg:morse} since there are no $E$-critical
vertices in degree 1 and higher.
\end{proof}

Now, we will construct a matching on $\cpx{G}$, viewed as a based
complex of $S$-modules, that will give us a splitting homotopy $\phi$
of $\cpx{G}$. Using the homotopy $\phi$ we can then describe the
minimal resolution of $M$.
Given a presentation~(\ref{eq:in_lin_syz}), we define for
a basis element $g_i$ of the free module $\bigoplus_j S\cdot g_j$,
its critical and non-critical indices by
\[
\crit(g_i) =
\{ j \mid x_j \, g_i \in\initial_{\prec}(\kernel \canonicalsurj) \},
\quad
\ncrit(g_i) = [n] \setminus \crit(g_i).
\]
Suppose $M$ has initially linear syzygies with a presentation as in
(\ref{eq:in_lin_syz}), we then have a $k$-basis for $M$,
(abusing notation by writing $x^{\alpha} g_i$ for
$\canonicalsurj(x^{\alpha} \, g_i)$),
\[
\{ x^{\alpha} \, g_i \mid \supp \alpha \subseteq \ncrit
g_i \}.
\]
We consider the resolution $\cpx{F}$ as a based complex of $S$-modules
via the decomposition
\[
F_n \simeq \bigoplus_{\substack{I, \alpha, i\\\card{I} =n,  \supp \alpha \subseteq  \ncrit g_i}}
S \cdot e_I x^{\alpha} g_i.
\]
For each term $m$ in $\bigoplus_{i} S \cdot g_i$, we define a subset of
the vertices of $\digraph{\cpx{F}}$
\[
V_m =  \{S \cdot e_I x^{\alpha} g_j \mid
x_I  x^{\alpha} g_j = m \},
\]
and for each such $m$, we will now define a partial matching $E_m$ on
the digraph $\digraph{\cpx{F}}|_{V_{m}}$ by
\[
E_m = \{S \cdot e_I  x^{\alpha} g_j
\rightarrow
 S \cdot e_{I\smallsetminus i} x_i x^{\alpha} g_j
\mid
i = \max ((I \cup \supp \alpha) \cap \ncrit g_j)
\}.
\]

\begin{Lem}
  The set $E = \bigcup_m E_m$ is a Morse matching on the digraph
  $\digraph{\cpx{F}}$.
  The set of unmatched vertices consists of all $S \cdot e_I g_j$ where
  $I \subseteq \crit g_j$.
\end{Lem}
\begin{proof}
  It is clear that $E_m$ is a partial matching on $\digraph{\cpx{F}}|_{V_m}$
  and along the same lines as in the proof of Lemma~\ref{lemma:gen_koszul},
  if $\alpha \rightarrow \beta \in E_m$ for $\alpha, \beta \in V_m$
  then all $\gamma\in V_m$ such that $\gamma \neq \beta$ and
  with an edge $\alpha \rightarrow \gamma$ are unmatched. Thus,
  there is no oriented cycle in the
  finite digraph $\morsegraph{\cpx{F}}{E}|_{V_m}$,
  and  by \cite[Lemma 1]{skoldberg:morse}, $E_m$ is a Morse matching
  on  $\digraph{\cpx{F}}|_{V_m}$. Whenever $\alpha \rightarrow \beta \in E$
  with $\alpha \in V_{m_1}$ and $\beta \in V_{m_2}$ we have
  $m_2 \preceq m_1$, so by  \cite[Lemma 7]{skoldberg:morse}, $\bigcup_m E_m$
  is a Morse matching on the full graph $\digraph{\cpx{F}}$
\end{proof}

With the above result, we can now define an $S$-linear splitting
homotopy $\phi$ on the resolution $\cpx{F}$ that allows us to
construct a smaller resolution $\cpx{G}$.
We will first give an explicit recursive definition of $\phi$.
In the definition we use the following notation due to Knuth:
When $P$ is some proposition, then $[P]=1$ if $P$ is true, and
$[P]=0$ if $P$ is false.
Let us start by defining the two helper functions $\iota$ and $\phi_0$ by
\begin{align*} \label{eq:iota_def}
  \iota(\alpha) &= \max (\supp \alpha) \\
  \phi_0(e_{I}x^{\alpha}g_j) &=
  [\iota(\alpha) > \max (I\cap\ncrit g_i)]
  \sgn{\iota(\alpha)}{I\cup\iota(\alpha)} \cdot
  e_{I\cup \iota(\alpha)} \frac{x^{\alpha}}{x_{\iota(\alpha)}}g_j
\end{align*}
then
\begin{equation*}
    \phi(e_{I}x^{\alpha}g_j) =
      \phi_0(e_{I}x^{\alpha}g_j)
      - \phi
      \left(
        d \phi_0 (e_{I}x^{\alpha}g_j) - e_{I}x^{\alpha}g_j
     \right).
\end{equation*}
Let $\pi$ be the projection
\[
\pi: \bigoplus_{I, \alpha, j} S \cdot e_I x^{\alpha} g_j \longrightarrow
\bigoplus_{\substack{I,j\\I \subseteq\crit(g_j)}} S \cdot e_I g_j.
\]


Now we can define a complex $\cpx{G}$ by letting
\[G_n = \bigoplus_{\substack{I, j\\I \subseteq\crit(g_j), \card{I} = n}} S \cdot e_I g_j\]
and defining the differential by
$d_G = \pi(d_F - d_F \phi d_F)$; we can then formulate the following
result that generalises work of Batzies and
Welker~\cite{batzies_welker:morse_cellular} and Herzog and
Takayama~\cite{herzog_takayama:mapping_cones}.
\begin{Thm}\label{thm:min_res}
Let $M$ be a module with initially linear syzygies, then
$\cpx{G}$ is the minimal free resolution of $M$.
\end{Thm}

\begin{proof}
  Follows from applying \cite[Theorem 2]{skoldberg:morse} to the
  matching $E$ and the resolution $\cpx{F}$.
\end{proof}

From Theorem~\ref{thm:min_res}, we can now immediately deduce the
following two corollaries giving the projective dimension and
Castelnuovo--Mumford regularity of a module with initially linear
syzygies.
\begin{Cor}
  If $M$ is minimally generated by $m_1,\dots, m_r$, and $M$ has
  initially linear syzygies then
  $\mathrm{p. dim} M = \max_{j} \card{\crit e_j}$
\end{Cor}
\begin{Cor}
  If $M$ is minimally generated by $m_1,\dots, m_r$, and $M$ has
  initially linear syzygies then
  $\mathrm{reg} M = \max_{i,j} (\deg m_i - \deg m_j)$
\end{Cor}
By specialising the last  corollary we get:
\begin{Cor} \label{cor:linres}
  A graded module generated by homogeneous elements of the same degree
  with initially linear syzygies has a linear resolution.
\end{Cor}

Another, non-recursive,  way of describing the differential in
$\cpx{F}$
is in terms of \emph{reductions}, as used by J\"ollenbeck and
Welker~\cite{joellenbeck_welker:morse}
in their description of the differential in the Anick resolution.
In the Morse digraph $\morsegraph{\cpx{G}}{E}$, we define an
\emph{elementary reduction path} to be a zig-zag path of length 2
\[
\alpha_0 \rightarrow \beta \rightarrow \alpha_1,
\]
where $\alpha_0, \alpha_1$ are in degree $i$, and $\beta$ is in degree
$i-1$ if $\alpha_0$ is in $E^{0}\cup E^{+}$, and in degree $i+1$ if
$\alpha \in E^{-}$. To such a path we assign the corresponding
\emph{elementary reduction} which is the map
\[
\rho_{\alpha_1,\alpha_0} =
\begin{cases}
  -d^{-1}_{\beta,\alpha_1} \circ d_{\beta,\alpha_0},
  & \quad \text{if } \alpha_0 \in E^{0} \cup E^{+}, \\
  -d_{\alpha_1,\beta} \circ d^{-1}_{\alpha_0,\beta},
  & \quad \text{if } \alpha_0 \in E^{-}.
\end{cases}
\]
The matching condition implies that there
is at most one elementary reduction path from $\alpha_0$ to $\alpha_1$.

In general, we let a \emph{reduction path} be a composition of zero or
more elementary reduction paths
\[
p =
\alpha_0 \rightarrow \beta_1 \rightarrow \alpha_1 \rightarrow \cdots
\rightarrow \beta_n \rightarrow \alpha_n, \qquad n \geq 0,
\]
and the corresponding reduction to be the composition
\[
\rho_{p} = \rho_{\alpha_n,\alpha_{n-1}}     \circ
           \rho_{\alpha_{n-1},\alpha_{n-2}} \circ
           \cdots                        \circ
           \rho_{\alpha_1,\alpha_0}.
\]
We denote the set of all reduction paths from $\alpha$
to $\beta$ by $\redpaths{\alpha}{\beta}$.
From~\cite[Lemma 5]{skoldberg:morse} we can now conclude that for a basis
element $\resbasis{I}{\alpha}$, we can write
\begin{equation}
  \label{eq:reduction_diff}
  d_{G}(\resbasis{I}{\alpha}) =
  \sum_{\substack{\resbasis{K}{\delta}\\K \subseteq \crit g_{\delta}}}
  \sum_{\substack{\bigbasis{J}{\beta}{\gamma}\\\resbasis{I}{\alpha}\rightarrow\bigbasis{J}{\beta}{\gamma}}} \sum_{p\in\redpaths{\bigbasis{J}{\beta}{\gamma}}{\resbasis{K}{\delta}}}
  \rho_{p}d_{\bigbasis{J}{\beta}{\gamma},\resbasis{I}{\alpha}}(\resbasis{I}{\alpha}).
\end{equation}

In our case we can see that we can divide the elementary reduction
paths originating in a vertex in $E^{-}$ into two types. We say that
an elementary reduction path is of type 1 when it is of the form
\[
  \bigbasis{I}{\alpha}{k} \rightarrow
  \altbigbasis{I\cup j}{\frac{x^{\alpha}}{x_j}}{g_k} \rightarrow
  \altbigbasis{(I\cup j)\smallsetminus i}{\frac{x^{\alpha}}{x_j}}{g_k},
 \]
where $j$ is the maximal element in
$(\supp \alpha \cup I) \cap \ncrit g_k$. The corresponding reduction
map is
\[
\rho(\bigbasis{I}{\alpha}{k}) =
  \sgn{j}{I \cup j}\sgn{i}{I\cup j} x_i
  \altbigbasis{(I\cup j)\smallsetminus i}{\frac{x^{\alpha}}{x_j}}{g_k}.
\]
An elementary reduction path is of type 2 if it is of the form
\[
  \bigbasis{I}{\alpha}{k} \rightarrow
  \altbigbasis{I\cup j}{\frac{x^{\alpha}}{x_j}}{g_k} \rightarrow
  \altbigbasis{(I\cup j)\smallsetminus i}{x^{\beta}}{g_l},
\]
where $x^{\beta}g_l$ appears with nonzero coefficient
$\lambda_{i,j,k,\alpha,\beta,l}$ in $\nf((x_ix^{\alpha}/x_j)g_k)$, and $j$ is the
maximal element in
$(\supp \alpha \cup I) \cap \ncrit g_k$. The corresponding reduction
map is
\[
\rho(\bigbasis{I}{\alpha}{k}) =
  -\sgn{j}{I \cup j}\sgn{i}{I\cup j}\lambda_{i,j,k,\alpha,\beta,l}
  \altbigbasis{(I\cup j)\smallsetminus i}{x^{\beta}}{g_l}.
\]

We will now see that in well-behaved cases, there is a more
explicit description of the differential.
\begin{Def}
  If $M$ has a presentation with initial linear syzygies such that
  for every generator $g$ we have that
  \[
  \crit(\nf(x_i g)) \subseteq \crit(g),
  \]
  we say that $M$ is crit-monotone. (We interpret
  $\crit(\sum_{j\in J}p_je_j)$ as $\bigcup_{j\in J}\crit(e_j)$ if there
  are no redundant terms in the sum.)
\end{Def}
It is easy to see that stable monomial ideals are crit-monotone, and
in~\cite{herzog_takayama:mapping_cones}, Herzog and Takayama prove
that matroidal ideals are crit-monotone, and they show that the
differential in the minimal resolution of crit-monotone monomial
ideals is of Eliahou--Kervaire type. Below, we will generalise their
result to general crit-monotone modules with initially linear
syzygies.

For a basis element $\altbasis{I}{g}$ with $I = \{i_1,\ldots, i_n\}$ we
define maps $d^{L}_{j}$, $d^{R}_{j}$, the left and right components of
the differential, for $1 \leq j \leq n$ by
\[
d^{L}_{j} (\altbasis{I}{g}) =
x_{i_{j}} \altbasis{I \smallsetminus i_j }{g}
\]
and
\[
d^{R}_{j} (\altbasis{I}{g}) =
\sum_k [I \smallsetminus i_j \subseteq \crit(g_k) ] \,
p_k \, \resbasis{I \smallsetminus i_j }{k},
\]
where $\nf(x_{i_{j}} g) = \sum_k p_k g_k$ for $p_k \in S$ and $g_k$ a
basis element, and we again make use of the Knuth notation.

\begin{Thm}
  Let $M$ be a crit-monotone module with initially linear syzygies,
  then the differential $d$ in the minimal resolution $\cpx{G}$ is given
  in degree $n$ by $d = \sum_{i=1}^{n}(-1)^{i-1}(d^{L}_{i} - d^{R}_{i})$
\end{Thm}

\begin{proof}
Consider an elementary reduction path,
$\bigbasis{I}{\alpha}{i} \rightarrow \bigbasis{J}{\beta}{j}$,
it is then easy to see that
\[
\card{J \cap \crit g_j} \leq \card{I \cap \crit g_i}.
\]
This observation together with the fact that we are only interested
in reduction paths ending in vertices $\resbasis{J}{}$ with
$J\subseteq \crit g$, gives that the only non-trivial reduction paths
appearing in the sum~(\ref{eq:reduction_diff}) are concatenations of
elementary reduction paths of type 1 that are of the form
\[
  \bigbasis{I}{\alpha}{k} \rightarrow
  \altbigbasis{I\cup j}{\frac{x^{\alpha}}{x_j}}{g_k} \rightarrow
  \altbigbasis{I}{\frac{x^{\alpha}}{x_j}}{g_k},
\]
which have coefficient $x_j$ in the corresponding reduction.
The reduction paths of length 0 that appear in the
sum~(\ref{eq:reduction_diff}) thus contribute
$\sum_{i=1}^{n}(-1)^{i-1}d^{L}_{i}$ to the differential, and the
concatenations of reduction paths mentioned above contribute
$-\sum_{i=1}^{n}(-1)^{i-1}d^{R}_{i}$, which sums to the desired formula.
\end{proof}


Herzog and Hibi have introduced the concept of componentwise linear
ideals \cite{herzog_hibi:componentwise_linear}.
For a graded module $M$ we let $\component{M}{j}$ be the module
generated by all homogeneous elements of degree $j$ in $M$; using
this notation we say that $M$ is \emph{componentwise linear} if
the modules $\component{M}{j}$ have linear resolutions for all $j$.

\begin{Thm} \label{thm:componentwise}
  If the finitely generated graded module $M$ has initially linear
  syzygies, then  $\component{M}{d}$ has initially linear syzygies for
  all $d$.
\end{Thm}

\begin{proof}
  By the hypothesis, there are homogeneous elements $m_1, m_2, \dots ,m_g$,
  that generate $M$, and a presentation
  \[
  0 \longrightarrow
  \kernel \canonicalsurj \longrightarrow
  \bigoplus_{i=1}^{g} S \cdot e_i \stackrel{\canonicalsurj}{\longrightarrow}
  M \longrightarrow 0
  \]
  with $\canonicalsurj(e_i) = m_i$,
  together with a term order `$\prec$' on
  $F = \bigoplus_{i=1}^{g} S \cdot e_i$ such
  that $\kernel \canonicalsurj$ has a Gr\"obner basis $G$ consisting of
  initially linear terms.

  We will now construct an explicit Gr\"obner basis $G_{d}$ for
  the syzygies of $\component{M}{d}$, and we will start by constructing
  a presentation of $\component{M}{d}$. This module is minimally generated
  by the images of the elements $x^{\alpha}e_i$ of degree $d$ in $M$
  that are irreducible with respect to $G$, and has a $k$-basis given by
  all the irreducible elements $x^{\alpha}e_i$ where $\deg e_i \leq d$.
  Thus, we consider the free module with basis elements $t^{\alpha}e_i$
  where $\deg t^{\alpha} + \deg e_i = d$ and $x^{\alpha}e_i$ irreducible with
  respect to $G$, and define the map
  $\canonicalsurj_d$ by
  $\canonicalsurj_d(t^{\alpha}e_i)=x^{\alpha}\canonicalsurj(e_i)=x^{\alpha}m_i$.
  We now have the presentation
  \[
  0 \longrightarrow \kernel \canonicalsurj_{d} \longrightarrow
  \bigoplus S \cdot t^{\alpha}e_i
  \stackrel{\canonicalsurj_d}{\longrightarrow} \component{M}{d} \longrightarrow 0,
  \]
  where the direct sum ranges over all $(\alpha, i)$ such that
  $x^{\alpha} e_i$ is irreducible with respect to $G$ and has degree $d$.

  Now, we  define the term order $\prec_d$ on
  the free module $F_d = \bigoplus S \cdot t^{\alpha}e_i$ by letting
  $x^{\alpha} \cdot t^{\beta}e_i \preceq_{d} x^{\gamma} \cdot t^{\delta} e_j$
  if
  \begin{align*}
    &x^{\alpha+\beta}e_i \prec x^{\gamma+\delta}e_j, \text{ or} \\
    &x^{\alpha+\beta}e_i = x^{\gamma+\delta}e_j, \text{ and }
    t^{\beta} \revlexlesseq t^{\delta}.
  \end{align*}
  Let $G_d$ be a  set consisting  of two types of elements.
  First, for every element
  $x_a \cdot e_i -\sum_{j} g_j \cdot e_j \in G$, and every
  monomial $x^{\alpha}$ of degree $d - \deg e_i$
  such that $\supp \alpha \subseteq \ncrit e_i$ we consider
  $x_ax^{\alpha} e_i - \sum_{j}\sum_k h_{j,k} \cdot e_k$
  such that
  $\nf(x^\alpha g_j e_j) = \sum_k h_{j,k} \cdot e_k =
     \sum_{k,l} c_{j,k,l}x^{\beta_{j,k,l}}e_k$
  and we let
  $x_a \cdot t^{\alpha}e_i -
     \sum_{j,k,l} c_{j,k,l}x^{\beta'_{j,k,l}}\cdot t^{\beta''_{j,k,l}}e_k \in G_d$,
  where $x^{\beta_{j,k,l}} = x^{\beta'_{j,k,l}}x^{\beta''_{j,k,l}}$ in such a way
  that $\deg x^{\beta''_{j,k,l}}e_j = d$.
  Second, for every $e_i$ and every $\alpha$ such that
  $\supp \alpha \subseteq \ncrit e_i$
  and $\deg x^{\alpha} = d - \deg e_i - 1$, we have the elements
  $ x_a \cdot t_{b} t^{\alpha}e_i -  x_b \cdot t_{a} t^{\alpha} e_i$
  for each $a,b \in \ncrit(e_i)$.

  The claim is now that $G_d$ is a Gr\"obner basis for
  $\kernel \canonicalsurj_d$. To start with, it is clear that $G_d$
  lies in the kernel of $\canonicalsurj_d$, so it is sufficient to
  prove that for all $i \geq d$, there is a bijection
  between the terms of degree $i$ in $F_d$ which are irreducible with
  respect to $G_d$, and the terms of degree $i$ in $F$ which are
  of the form $x^{\alpha}e_j$ where $\deg e_j \leq d$ and are
  irreducible with respect to $G$.
  The irreducible terms in $F_d$ of degree $i$ are all
  $x^{\alpha}t^{\beta}e_j$ such that
  $\supp (\alpha+\beta) \subseteq \ncrit e_j$ and
  $\max \supp \alpha \leq \min \supp \beta$, and from this we can conclude
  that the map $x^{\alpha}t^{\beta}e_i \mapsto x^{\alpha+\beta}e_i$
  is a bijection.
\end{proof}

An immediate consequence of the preceding theorem is the following
corollary that generalises \cite{sharifan_varbaro:linear_quotients},
where it is shown that a homogeneous ideal with linear quotients is
componentwise linear.

\begin{Cor}
  A graded module with initially linear syzygies  is componentwise
  linear.
\end{Cor}
\begin{proof}
  Follows directly from Theorem~\ref{thm:componentwise} and
  Corollary~\ref{cor:linres}.
\end{proof}

\section{A contracting homotopy}

We will now consider the minimal resolution $\cpx{G}$ of a quadratic
monomial ideal $M$ with initially linear syzygies as a based complex
of $k$-modules by the natural decomposition
\[
G_n = \bigoplus_{\substack{\alpha, I, j\\I \subseteq \crit g_j}} k \cdot x^{\alpha} e_I g_j.
\]
From a Morse matching on the digraph $\digraph{\cpx{G}}$ we will
construct a contracting homotopy $c$ on $\cpx{G}$, that is a $k$-linear map
satisfying
\[
dc + cd = 1 - \canonicalsurj.
\]
The contracting homotopy will be
used in the next section to show the existence of a DGA structure on
$\cpx{G}$ in the case when $M$ is a stable monomial ideal or a
squarefree matroidal monomial ideal.
The matching is constructed as follows.
For $\beta \in \nat^n$, let $V_{\beta,j}$ be the subset of $V$ consisting
of all $x^{\alpha} e_I  g_j$ such that
$\deg_{\nat^n} x^{\alpha}e_I = \beta$.
Exactly as in the proof of Lemma~\ref{lemma:gen_koszul},
we construct a partial  matching $B_{\beta,j}$ on
$\digraph{\cpx{G}}|_{V_{\beta,j}}$
\[
B_{\beta,j} =
\{x^{\alpha}\resbasis{I}{j}
\rightarrow
x^{\alpha}x_{i} \resbasis{I\smallsetminus i}{j}
\mid
i = \min(\supp \alpha \cup I)\cap \crit g_j \}.
\]
\begin{Lem}
  The set $B = \bigcup_{\beta, j}B_{\beta,j}$  is a Morse matching on
  $\digraph{\cpx{G}}$.
\end{Lem}
\begin{proof}
The same argument as for Lemma~\ref{lemma:gen_koszul} shows that
each $B_{\beta,j}$ is a Morse matching on $\digraph{\cpx{G}}|_{V_{\beta,j}}$,
and we note that whenever there is an edge from a vertex in
$V_{\beta,j}$ to a vertex in a different set $V_{\gamma,k}$ we have
$x^{\gamma}e_k \prec x^{\beta}e_j$.
\end{proof}
Let
\begin{align*}
  \tilde{\iota}(\alpha,j) &= \min(\supp \alpha \cap \crit (g_j)), \\
  c_0(x^{\alpha} \cdot e_{I}g_j) &=
  [\tilde{\iota}(\alpha,j) < \min I]
  \frac{x^{\alpha}}{x_{\tilde{\iota}(\alpha,j)}} e_{I\cup \tilde{\iota}(\alpha,j)} g_j.
\end{align*}

We can now define the map $c$ by
\begin{equation*}
  c(x^{\alpha} \cdot e_{I}g_j) =
  c_0(x^{\alpha} \cdot e_{I}g_j)
  -
  c
  \left(
    dc_0(x^{\alpha} \cdot e_{I}g_j) - x^{\alpha} \cdot e_{I}g_j
  \right).
\end{equation*}

A consequence of the above lemma is:
\begin{Cor}
  The $k$-linear map $c$ is a contracting homotopy on $\cpx{G}$ such that
  $\im (c) $ is spanned by the elements
  \[
  \{x^{\alpha} e_I g_j \mid \min ((\supp x^\alpha \cup
  I) \cap \crit g_j) \in I\}.
  \]
\end{Cor}
\begin{proof}
  This follows from~\cite[Lemma 6]{skoldberg:morse}.
\end{proof}

For monomial ideals with crit-monotone presentations, we can say a bit more
about the contracting homotopy, by again using reductions for our
description. We will define the set of \emph{$c$-critical} indices of a basis
element $\resbasis{I}{j}$ by
\[
  \ccrit(\resbasis{I}{j}) =
  \{ i \mid i \in \crit{g_j}, i < \min I   \}.
\]
We have the following formula for the homotopy acting on a basis element:
\[
  c(x^{\alpha} \, \resbasis{I}{\beta}) =
  \sum_{x^{\delta}  \resbasis{J}{\gamma}\in B^{-}}
  \sum_{p \in
    \redpaths{x^{\alpha}  \resbasis{I}{\beta}}{x^{\delta} \resbasis{J}{\gamma}}}
  c_0 \rho_{p}(x^{\alpha} \resbasis{I}{\beta}).
\]
The composition with $c_0$ means that only elementary reductions paths
$\alpha_0 \rightarrow \beta_1 \rightarrow \alpha_1$ where $\alpha_1$ is
in $B^{-}$ will contribute to the result.
These reduction paths are of the form
\[
x^{\alpha} \, \resbasis{I}{\beta} \rightarrow
\frac{x^{\alpha}}{x_i} \, \resbasis{I \cup \{i\}}{\beta} \rightarrow
\frac{x^{\alpha}}{x_i}x^{\delta} \, \resbasis{I}{\gamma},
\]
where $i = \min (\ccrit(\resbasis{I}{\beta}) \cap \supp x^{\alpha})$,
$x^{\delta} g_{\gamma} = \nf(x_ig_{\beta})$, and $I \subseteq \crit(g_{\gamma})$.
We can also note that for each $k$-basis element
$x^{\alpha} \, \resbasis{I}{\beta}$, there is at most one elementary reduction
path emanating from it.
This means that the terms that occur in
$c(x^{\alpha} \, \resbasis{I}{\beta})$ are all of the form
$\frac{x^{\alpha}}{m}n \, \resbasis{i \cup I}{\gamma}$ where $m$ divides
$x^{\alpha}$,

We can now define a $k$-linear function $\rho$ by
setting
\[
  \rho(x^{\alpha} \, \resbasis{I}{\beta}) =
  \begin{cases}
    \frac{x^{\alpha}}{x_i} x^{\delta}\, \resbasis{I}{\gamma}, &\quad
    \text{if } J \neq \emptyset \text{ with } i = \min J, \\
    0, & \quad \text{if } J = \emptyset.
  \end{cases}
\]
where $J = \supp x^{\alpha}  \cap \ccrit(\resbasis{I}{\beta})$ and
$\nf(x_ig_{\beta}) = x^{\delta} g_{\gamma}$.

From these observations we can now deduce the following lemma.

\begin{Lem}\label{lem:homotopy}
  Let $M$ be a crit-monotone monomial ideal with initially linear
  syzygies. The contracting homotopy $c$ is given by
  \[
  c(x^{\alpha} \, \resbasis{I}{\beta}) =
  \sum_{j} c_0 \rho^{j} (x^{\alpha} \, \resbasis{I}{\beta}),
  \]
  where $\rho$ is defined as above.
\end{Lem}

\section{DGA structures on resolutions}

In this section we will construct a differential graded algebra
structure on the minimal resolution of $S/I$ where $I$ is either a
stable monomial ideal or a squarefree matroidal ideal. We can thereby
extend, with a simpler proof, the result of
Peeva~\cite{peeva:borelfixed} showing the existence of a DGA structure
on the minimal resolution of $S/I$ where $I$ is a stable ideal.

Let $\cpx{\tilde{G}}$ be the resolution of $S/I$ obtained by splicing the
resolution $\cpx{G} \rightarrow I$ with
$0 \rightarrow I \rightarrow S \rightarrow S/I \rightarrow 0$. Thus,
we have
\[
\cpx{\tilde{G}}: \quad
0 \longrightarrow
G_n \longrightarrow
\cdots \longrightarrow
G_1 \longrightarrow
G_0 \longrightarrow
S \longrightarrow
0,
\]
and we can extend the contracting homotopy $c$ defined on $\cpx{G}$ to
$\cpx{\tilde{G}}$ by setting $c(x^{\alpha}) = x^{\alpha - \beta} g_{\beta}$,
where $x^{\beta}$ is the smallest 
monomial generator of $I$ with respect to
$\prec$ that divides $x^{\alpha}$. It is easy to see that with this
definition, $c^2(x^{\alpha}) = 0$, and therefore $c^2 = 0$.

We will now define a map
$\mu : \cpx{\tilde{G}} \otimes_S \cpx{\tilde{G}}
\longrightarrow \cpx{\tilde{G}}$
so $\mu$ is the multiplication in a DGA structure on
$\cpx{\tilde{G}}$. The technique we will use to establish this result
rests upon the following Lemma, which is a special case
of~\cite[Theorem IX.6.2]{maclane:homology}.

\begin{Lem}\label{lem:maclane}
  Suppose that $\cpx{X}$ and $\cpx{Y}$ are complexes of $S$-modules,
  where $X_n = S\otimes_k V_n$ and $Y_n = S \otimes_k W_n$ for
  $k$-spaces $V_n$ and $W_n$, $n \geq 0$. Furthermore, suppose that
  $\cpx{Y}$ is acyclic, with a contracting homotopy $c$ satisfying
  $c^2 =0$. Then, every $S$-linear map $\phi_0:X_0 \longrightarrow Y_0$
  has a unique lifting to a chain map $\phi: \cpx{X} \longrightarrow \cpx{Y}$
  satisfying $\phi(V_n) \subseteq \im(c)$. This map is defined
  inductively by
  \[
  \phi_{n+1}(\bar{x}) = c\phi_{n}d(\bar{x}), \qquad \bar{x} \in V_{n+1}.
  \]
\end{Lem}
We call elements of $V_n \subseteq S\otimes_k V_n$ \emph{reduced}; the
reduced elements of $\cpx{\tilde{G}}$ are thus the $S$-basis elements
of $\tilde{G}_n$.

Thus, we define our map $\mu$ on the reduced elements $\bar{x}$ and
$\bar{y}$ of degrees $m$ and $n$ respectively by:
\begin{equation}
  \mu(\bar{x}\otimes\bar{y}) = c \mu d(\bar{x}\otimes\bar{y}) =
  c \mu (d(\bar{x}) \otimes \bar{y}) +
  (-1)^{m} c \mu(\bar{x} \otimes d(\bar{y})).
\end{equation}

Now, consider the composition
\[
\cpx{\tilde{G}}
\stackrel{\simeq}{\longrightarrow}
S \otimes_S \cpx{\tilde{G}}
\stackrel{\iota \otimes 1}{\longrightarrow}
\cpx{\tilde{G}} \otimes_S \cpx{\tilde{G}}
\stackrel{\mu}{\longrightarrow}
\cpx{\tilde{G}},
\]
which is the identity in degree 0, and since
$\mu(1\otimes\bar{x}) \in \im(c)$; by Lemma~\ref{lem:maclane}, this is
then the identity in all degrees, so $1 \in \tilde{G}_0$ is a
multiplicative identity element.

Furthermore, letting $\tau$ be the twist morphism,
$\tau(x\otimes y) = (-1)^{mn} y \otimes x$ where
$x$ and $y$ are homogeneous of degrees $m$ and $n$ respectively,
we have that $\mu$ and $\mu \circ \tau$ both are chain maps
$\cpx{\tilde{G}} \otimes \cpx{\tilde{G}} \longrightarrow \cpx{\tilde{G}}$
that in degree 0 are given by
$\mu(1\otimes 1) = 1 = \mu\circ\tau(1\otimes 1)$. Since for reduced
elements $\bar{x}$ and $\bar{y}$ we have that
$\mu \circ \tau(\bar{x} \otimes \bar{y}) \in \im(c)$,
Lemma~\ref{lem:maclane} gives that $\mu = \mu \circ \tau$, so $\mu$ is
graded commutative.
Thus, to show that $\mu$ gives a DGA structure to $\cpx{\tilde{G}}$,
it remains to show that $\mu$ is associative, that is, that
$\mu(1 \otimes \mu) = \mu(\mu \otimes 1)$.


Recall that for a basis element $e_{I}g_{j}$ we have
$
\ccrit(e_{I}g_j) =
\{ i \mid i \in \crit{g_j}, i < \min I   \},
$
and we now extend this to the whole of $\cpx{\tilde{G}}$ by letting
\[
\ccrit(\sum_{i} p_i \cdot e_{I_i}g_{j_i}) =
\bigcup_{i} \ccrit(e_{I_i}g_{j_i})
\]
where we have no redundant terms in the sum.
We are now in a position to formulate and prove the lemma that we will
use to show associativity. We will in the following write $x \star y$
for $\mu(x \otimes y)$.
\begin{Lem}\label{lem:intersection}
  If, for all basis elements $e_{I}g_{i}, e_{J}g_{j}$,
  we have
  \begin{equation}\label{eq:intersection}
  \ccrit (e_{I}g_{i} \star e_{J}g_{j}) \subseteq
  \ccrit(e_{I}g_{i}) \cap  \ccrit(e_{J}g_{j}),
  \end{equation}
  then $\star$ is associative.
\end{Lem}

\begin{proof}
  Since $1 \star x = x $ for all $x \in \cpx{\tilde{G}}$, we only have
  to show that
  \begin{equation} \label{eq:assoc1}
  e_I g_i \star (e_J g_j \star e_K g_k) =
  (e_I g_i \star e_J g_j) \star e_K g_k
  \end{equation}
  for all basis elements $e_I g_i$, $e_J g_j$ and $e_K g_k$,
  and this now follows if we can show that
  \begin{equation} \label{eq:assoc2}
  e_I g_i \star (e_J g_j \star e_K g_k) \in \im c,
  \end{equation}
  holds for all $e_I g_i$, $e_J g_j$ and $e_K g_k$,
  since then by the graded commutativity of $\star$ we would
  also get that
  \begin{equation} \label{eq:assoc3}
  (e_I g_i \star e_J g_j) \star e_K g_k \in \im c
  \end{equation}
  and by Lemma~\ref{lem:maclane} we can conclude
  that they are equal.
  Now, suppose that $x^{\alpha} \, \resbasis{L}{l}$ occurs as a term in
  $e_J g_j \star e_K g_k$, then, by the condition of the lemma, no
  variable occurring in $x^{\alpha}$ will be c-critical in
  $\resbasis{I}{i}\star\resbasis{L}{l}$, and thus
  (\ref{eq:assoc2}) follows, and the proof is complete.
\end{proof}

We will now in a series of lemmata show that the conditions of
Lemma~\ref{lem:intersection} are satisfied for the minimal resolution
of stable and squarefree matroidal ideals. We start with the case
where one of the basis elements have minimal degree.

\begin{Lem}\label{lem:base_case_stable}
  Let $g_\alpha$ and $\resbasis{I}{\beta}$ be basis elements in the
  minimal resolution of a stable ideal. We then have
  \[
  \ccrit(c(x^\alpha \, \resbasis{I}{\beta})) \subseteq
  \ccrit(\modbasis{\alpha}).
  \]
\end{Lem}
\begin{proof}
  By Lemma~\ref{lem:homotopy}, we have that
  \[
  c(x^\alpha \, \resbasis{I}{\beta}) =
  \sum_{j} c_0(x^{\alpha_j} \, \resbasis{I}{\beta_j})
  \]
  where $\alpha_j$ and $\beta_j$ satisfy
  $x^{\alpha_j} = x^{\alpha}v_j/u_j$ and $v_j g_{\beta_j} = \nf(u_jg_{\beta})$
  for some monomial $u_j$ of degree $j$ dividing $x^{\alpha}$.
  Now, if $c_0(x^{\alpha_j} \, \resbasis{I}{\beta_j}) \neq 0$, then
  by the crit-monotonicity of the stable ideals,
  $c_0(x^{\alpha_j} \, \resbasis{I}{\beta_j}) = x^{\alpha_j}/x_k
  \, \resbasis{I\cup k}{\beta_j}$, for a $k \in \supp (x^{\alpha}/u_j)$, and thus
  \[
  \ccrit(x^{\alpha_j}/x_k \, \resbasis{I\cup k}{\beta_j}) =  [1, k -1]
  \subseteq \crit(g_{\alpha}).
  \]
\end{proof}
Before showing the corresponding result for squarefree matroidal
ideals, we have to say something about the term order we use.
\begin{Lem}
  Let $M$ be a squarefree matroidal ideal in $S$, then $M$ is shellable
  with respect to a lexicographic ordering.
\end{Lem}
\begin{proof}
  Essentially the same as the proof given by Herzog and Takayama
  \cite[Lemma 1.3]{herzog_takayama:mapping_cones} for the revlex
  order.
\end{proof}
From this we can conclude that the order given by
$m\, g_{\alpha} \prec n \, g_{\beta}$ whenever
$x^\alpha \lexless x^\beta$ gives us initially linear syzygies.

\begin{Lem}\label{lem:base_case_matroid}
  Let $g_\alpha$ and $\resbasis{I}{\beta}$ be basis elements in the
  minimal resolution of a squarefree matroidal ideal. We then have
  \[
  \ccrit(c(x^\alpha \, \resbasis{I}{\beta})) \subseteq
  \ccrit(\modbasis{\alpha}).
  \]
\end{Lem}
\begin{proof}
  Again, by Lemma~\ref{lem:homotopy}, we have that
  \[
  c(x^\alpha \, \resbasis{I}{\beta}) =
  \sum_{i} c_0(x^{\alpha_i} \, \resbasis{I}{\beta_i})
  \]
  where $\alpha_j$ and $\beta_j$ satisfy
  $x^{\alpha_j} = x^{\alpha}v_j/u_j$ and $v_j g_{\beta_j} = \nf(u_jg_{\beta})$
  for some monomial $u_j$ of degree $j$ dividing $x^{\alpha}$.
  Now assume that $j \in \crit(g_{\beta_k})$ for some $j$ such that
  $j$ is less than all elements in
  $\crit(g_{\beta_k}) \cap \supp \frac{x^{\alpha}v_k}{u_k}$.
  From the observation that $x^{\alpha_i} x^{\beta_i} = x^{\alpha}x^{\beta}$
  for all $i$ we can conclude that
  \[
  \nf (x_j x^{\alpha_k}\modbasis{\beta_k}) =
  \nf(x_j x^{\alpha}\modbasis{\beta}).
  \]
  We can reduce $x_j x^{\alpha_k}g_{\beta_k}$ to $x_l x^{\alpha_k}g_{\beta_k'}$
  for some $l > j$, and there cannot be any later reduction of the
  form $x_m g_{\sigma} \rightarrow x_j g_{\sigma'}$ in a chain of
  reductions starting in $x_l x^{\alpha_k}g_{\beta_k'}$,
  since that would imply that $m < j$, and that $x_m \in \crit(g_{\beta_k})$,
  and then we would have $m \in \supp x^{\alpha_k}$ which contradicts the
  choice of $j$. Thus we have for
  $x^{\sigma}g_{\tau} = \nf(x_j x^{\beta}g_{\alpha})$ that $j \in \supp \tau$, and
  thus, by the crit-monotonicity, $j \in \crit(g_{\alpha})$.
\end{proof}

\begin{Lem}\label{lem:base_case_common}
  Let $g_\alpha$ and $\resbasis{I}{\beta}$ be basis elements in the
  minimal resolution of a stable ideal or a squarefree matroidal
  ideal. We then have
  \[
  \ccrit(c(x^\alpha \, \resbasis{I}{\beta})) \subseteq
  \ccrit(\resbasis{I}{\beta}).
  \]
\end{Lem}
\begin{proof}
  By Lemma~\ref{lem:homotopy}, we can see that the elements
  that appear in $c(x^\alpha \, \resbasis{I}{\beta})$ are of the form
  $x^{\gamma} \, \resbasis{i \cup I}{\delta}$, and by the crit-monotonicity
  we know that $\crit(g_{\delta}) \subseteq \crit(g_{\beta})$, hence the
  statement follows.
\end{proof}

We now turn to the case where the first basis elements in the product has
non-minimal degree.
\begin{Lem}\label{lem:induction_case}
  Let $\resbasis{I}{\alpha}$ and $\resbasis{J}{\beta}$ be
  two basis elements in $\cpx{\tilde{G}}$ with $I \neq \emptyset$.
  If the inclusion
  \[
  \ccrit (e_{K}g_{\gamma} \star e_{L}g_{\delta}) \subseteq
  \ccrit(e_{K}g_{\gamma}) \cap  \ccrit(e_{L}g_{\delta}),
  \]
  holds for all pairs of basis elements
  $\resbasis{K}{\gamma}$ and $\resbasis{L}{\delta}$
  where $\card{K} + \card{L} < \card{I} + \card{J}$, then
  \[
  \ccrit (c(d(\resbasis{I}{\alpha}) \star \resbasis{J}{\beta})) \subseteq
  \ccrit(\resbasis{I}{\alpha}) \cap \ccrit(\resbasis{J}{\beta}).
  \]
\end{Lem}
\begin{proof}
  By the crit-monotonicity, the differential $d$ can be written
  as
  \[ d = \sum_{j} (-1)^{j-1} d^L_j - \sum_{j} (-1)^{j-1} d^R_j,\]
  so  we consider the terms
  $c(d^{L}_{i}(\resbasis{I}{\alpha})\star \resbasis{J}{\beta})$ and
  $c(d^{R}_{i}(\resbasis{I}{\alpha})\star \resbasis{J}{\beta})$ separately.

  If $i=1$, we have
  \[
  c(d^{L}_{1}(\resbasis{I}{\alpha}) \star \resbasis{J}{\beta}) =
  c(x_i \, \resbasis{I\smallsetminus i}{\alpha} \star
  \resbasis{J}{\beta}),
  \qquad i = \min I.
  \]
  Suppose that $m \, \resbasis{K}{\gamma}$ is a term in the
  product
  $\resbasis{I\smallsetminus i}{\alpha} \star \resbasis{J}{\beta}$;
  by assumption, we then have that
  $\ccrit(\resbasis{K}{\gamma}) \subseteq
  \ccrit(\resbasis{I\smallsetminus i}{\alpha})$ and that
  $\supp m \cap \ccrit(\resbasis{K}{\gamma}) = \emptyset$.
  If $c(x_im \,\resbasis{K}{\gamma})$ is nonzero, then, by
  Lemma~\ref{lem:homotopy} we get that
  \[
  c(x_i m \, \resbasis{K}{\gamma}) = m \, \resbasis{i \cup K}{\gamma}
  \]
  and thus,
  \begin{align*}
  \ccrit(c(x_im \, \resbasis{K}{\gamma}))
  &= \ccrit(\resbasis{i\cup K}{\gamma}) \\
  &= \ccrit(\resbasis{K}{\gamma}) \cap [1,i-1] \\
  &\subseteq \ccrit(\resbasis{I\smallsetminus i}{\alpha}) \cap
   \ccrit(\resbasis{J}{\beta}) \cap [1,i-1] \\
  &= \ccrit(\resbasis{I}{\alpha}) \cap \ccrit(\resbasis{J}{\beta}).
  \end{align*}
  If $i > 1$, then we have
  \[
  c(d^{L}_{i}(\resbasis{I}{\alpha}) \star \resbasis{J}{\beta}) =
  c(x_i \, \resbasis{I\smallsetminus i}{\alpha} \star
  \resbasis{J}{\beta}),
  \qquad i > \min I.
  \]
  If $m \, \resbasis{K}{\gamma}$ occurs in the product
  $\resbasis{I\smallsetminus i}{\alpha} \star \resbasis{J}{\beta}$,
  then, since
  $i \not\in \ccrit(\resbasis{I\smallsetminus i}{\alpha})$,
  the hypothesis of the lemma gives us that
  $i \not\in \ccrit(\resbasis{K}{\gamma})$, and therefore
  $c(x_i m \, \resbasis{K}{\gamma}) = 0$.

  Now, we look at $d^{R}_i$; for all $i$ we have
  \[
  c(d^{R}_{i}(\resbasis{I}{\alpha}) \star \resbasis{J}{\beta}) =
  c(n \, \resbasis{I\smallsetminus i}{\gamma} \star \resbasis{J}{\beta})
  \]
  for some monomial $n$ with
  $\supp n \cap \ccrit(\resbasis{I\smallsetminus i}{\gamma}) = \emptyset$.
  Let $m \, \resbasis{K}{\delta}$ be a term in the product
  $\resbasis{I\smallsetminus i}{\gamma} \star \resbasis{J}{\beta}$,
  by assumption
  $\ccrit(m \, \resbasis{K}{\delta}) \subseteq
  \ccrit(\resbasis{I\smallsetminus i}{\gamma})$ and since
  $\supp m \cap \ccrit(\resbasis{K}{\gamma}) = \emptyset$ and
  $\supp n \cap \ccrit(\resbasis{K}{\gamma}) = \emptyset$,
  we can conclude that
  $c(mn \,\resbasis{K}{\gamma}) = 0$.
\end{proof}

\begin{Thm}
  The minimal resolution of $M$ where $M$ is a squarefree matroidal
  ideal or a stable ideal has a DGA structure.
\end{Thm}
\begin{proof}
  By Lemma~\ref{lem:intersection} it suffices to show that
  \begin{equation*}
  \ccrit(\resbasis{I}{\alpha}\star\resbasis{J}{\beta}) \subseteq
  \ccrit(\resbasis{I}{\alpha}) \cap \ccrit(\resbasis{J}{\beta})
  \end{equation*}
  holds for all basis elements $\resbasis{I}{\alpha}$ and
  $\resbasis{J}{\beta}$, and by the definition of the product, we thus
  need to verify the relations
  \begin{align}
  \ccrit(c(d(\resbasis{I}{\alpha}) \star\resbasis{J}{\beta})) & \subseteq
  \ccrit(\resbasis{I}{\alpha}) \cap \ccrit(\resbasis{J}{\beta}),
  \label{eq:inclusion1} \\
  \ccrit(c(d(\resbasis{J}{\beta}) \star\resbasis{I}{\alpha})) & \subseteq
  \ccrit(\resbasis{I}{\alpha}) \cap \ccrit(\resbasis{J}{\beta}).
  \label{eq:inclusion2}
  \end{align}
  We proceed by induction on $\card{I} + \card{J}$.
  If $\card{I} = \card{J} = 0$, (\ref{eq:inclusion1}) and
  (\ref{eq:inclusion2}) hold by
  Lemmas~\ref{lem:base_case_stable}, \ref{lem:base_case_matroid}
  and \ref{lem:base_case_common}.
  Now for the case $\card{I}+\card{J} > 0$,
  the inclusion (\ref{eq:inclusion1}) holds if $\card{I} = 0$ by
  Lemmas~\ref{lem:base_case_stable}, \ref{lem:base_case_matroid}
  and \ref{lem:base_case_common},
  and by induction and Lemma~\ref{lem:induction_case} if $\card{I} > 0$;
  and similarly for the inclusion (\ref{eq:inclusion2}).
\end{proof}

We will finish the paper by looking at the multiplicative structure
of the minimal resolution of a small squarefree matroidal ideal.

\begin{Ex}
  The Fano matroid is the matroid on the ground set
  $\{1,2,\ldots, 7\}$, where every 3-element set is a basis, except
  for the following sets:
  $\{1,2,3\}$, $\{1,4,7\}$, $\{1,5,6\}$, $\{2,4,6\}$, $\{2,5,7\}$,
  $\{3,4,5\}$, $\{3,6,7\}$.
  The Fano matroid it is often visualised using the following diagram
  
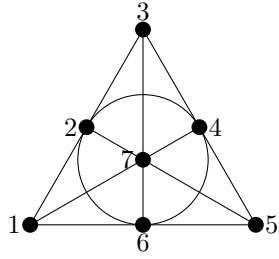
\begin{center}
    \begin{tikzpicture}[scale=3]

      \coordinate[label=left: $1$]  (A) at (-0.5, 0);
      \coordinate[label=right:$5$]  (B) at (0.5, 0);
      \coordinate[label=above:$3$]  (C) at (0, 0.866);
      \coordinate[label=below:$6$]  (D) at (barycentric cs:A=1 ,B=1 ,C=0);
      \coordinate[label=left: $2$]  (E) at (barycentric cs:A=1 ,B=0 ,C=1);
      \coordinate[label=right:$4$]  (F) at (barycentric cs:A=0 ,B=1 ,C=1);
      \coordinate[label=left: $7$]  (G) at (barycentric cs:A=1 ,B=1 ,C=1);

      \draw (A) -- (B) -- (C) -- cycle;
      \draw (A) -- (F);
      \draw (B) -- (E);
      \draw (C) -- (D);
      \node [draw] at (G) [circle through={(D)}] {};

      \foreach \point in {A,B,C,D,E,F,G}
        \fill (\point) circle (1pt);
    \end{tikzpicture}
  \end{center}
  where a curve is drawn through every 3-element circuit.  Let $I$ be
  the ideal in $S = k[x_1,\ldots, x_7]$ generated by the monomials
  corresponding to the bases in the Fano matroid. This ideal then has
  ${7 \choose 3} - 7 = 28$ generators, so for space reasons we will
  not describe the full multiplication table on the minimal resolution
  of $S/I$. Instead we will look at the resolution of $S'/J$, where
  the ideal $J$ is generated by the monomials in $I$ whose support
  is contained in
  $\{1,2,3,4\}$ and $S'$ is the polynomial ring $k[x_1,\ldots,x_4]$.
  This is then going to be a subalgebra of the minimal resolution of
  $S/I$. Thus, $J=(x_1x_2x_4, x_1x_3x_4, x_2x_3x_4)$, and we have
  the following basis elements in the resolution:
  \begin{center}
    \begin{tabular}{|r|l|}
      \hline
      Degree & Basis elements \\
      \hline
      0 & $1$\\
      1 & $g_{124}$, $g_{134}$, $g_{234}$\\
      2 & $\resbasis{2}{134}$, $\resbasis{1}{234}$ \\
      \hline
    \end{tabular}
  \end{center}

  Most of the products are either zero for degree reasons, or trivial
  due to multiplication by the identity element, so we are left with
  the following products to calculate: $g_{124} \star g_{134}$,
  $g_{124} \star g_{234}$ and $g_{134} \star g_{234}$:
  \begin{eqnarray*}
    g_{124} \star g_{134} &=& c(x_1x_2x_4 \, g_{134}) - c(x_1x_3x_4 \, g_{124}) \\
    &=&  x_1x_4 \, \resbasis{2}{134} - 0 \\
    &=&  x_1x_4 \, \resbasis{2}{134},
  \end{eqnarray*}
  \begin{eqnarray*}
    g_{124} \star g_{234} &=& c(x_1x_2x_4 \, g_{234}) - c(x_2x_3x_4 \, g_{124}) \\
    &=& x_2x_4 \, \resbasis{1}{234} - 0 \\
    &=& x_2x_4 \, \resbasis{1}{234},
  \end{eqnarray*}
  \begin{eqnarray*}
    g_{134} \star g_{234} &=& c(x_1x_3x_4 \, g_{234}) - c(x_2x_3x_4 \, g_{134}) \\
    &=& x_3x_4 \, \resbasis{1}{234} - x_3x_4 \, \resbasis{2}{134}.
  \end{eqnarray*}
\end{Ex}

\bibliographystyle{amsalpha}
\bibliography{bibfil}

\end{document}